\documentclass[11pt,reqno]{amsart}
\usepackage{amsthm,amssymb,amsmath}
\usepackage{textcomp}
\usepackage[foot]{amsaddr}
\usepackage{xcolor}
\usepackage{hyperref}
\usepackage[T1]{fontenc}
\usepackage{lmodern}
\usepackage{lmodern}
\hypersetup{
	colorlinks =true , 
	linkcolor =blue,
	urlcolor = magenta,
	citecolor =blue
}

\newtheorem{theorem}{Theorem}[section]
\newtheorem*{theorem*}{Theorem}
\newtheorem{lemma}{Lemma}[section]
\newtheorem{corollary}{Corollary}[section]
\newtheorem{proposition}{Proposition}[section]

\theoremstyle{definition}

\newtheorem{definition}{\bf Definition}[section]
\newtheorem*{definition*}{\bf Definition}

\newtheorem{example}{\bf Example}[section]

\newtheorem{remark}{Remark}
\newtheorem*{remark*}{Remark}

\newtheorem*{example*}{\bf Example}

\newcommand{\spec}{{\rm spec}}

\title[An infinite dimensional saddle point theorem and application]{An infinite dimensional saddle point theorem and application}

\begin{document}

\author{Fabrice Colin$^a$}

\address{$^a$Laurentian University, School of Engineering and Computer Science, Sudbury, Ontario, Canada}

\email{fcolin@laurentian.ca}

\author{Ablanvi Songo$^{1,b}$}

\address{$^b$Universit\'{e} de Sherbrooke, D\'{e}partement de math\'{e}matiques, Sherbrooke, Qu\'{e}bec, Canada}

\email{ablanvi.songo@usherbrooke.ca}
\footnote{Corresponding author}



    \begin{abstract}
	By using the $\tau$-topology of Kryszewski and Szulkin, we establish a natural new version of the Saddle Point Theorem for strongly indefinite functionals. The abstract result will be applied to study the existence of solutions to a strongly indefinite semilinear Schrödinger equation, where the associated functional is indefinite, that is, the functional is of the form $J(u) = \dfrac{1}{2} \langle Lu, u \rangle - \Psi(u)$ defined on a Hilbert
 space $X$, where $L : X \to X$ is a self-adjoint operator whose negative and positive eigenspaces are both infinite-dimensional.
\end{abstract}

\maketitle

\section{Introduction}
	In 1978, Rabinowitz \cite[Theorem 4.6]{Ra} established the following result. Let $X=Y\oplus Z$, where $X$ is a real Banach space and $Y\ne \{0\}$ and is finite dimensional. Let $I \in \mathcal{C}^1(X, \mathbb{R})$ be such that 
	\begin{enumerate}
		\item[$(I_1)$] there is a constant $\alpha$ and a bounded neighborhood $D$ of $0$ in $Y$ such that $\underset{u\in \partial D}{\max}\, I(u)\le \alpha$,
		\item[$(I_2)$] there is a constant $\beta> \alpha$ such that $\underset{u\in Z}{\inf}\, I(u)\ge \beta$.
	\end{enumerate}
	If $I$ satisfies the Palais-Smale condition $(PS)_c$ with 
	\begin{equation*}
		c:= \underset{\gamma\in \Gamma}{\inf} \underset{u\in \overline{D}}{\max}\,I(\gamma(u)),
	\end{equation*}
	where 
	\begin{equation*}
		\Gamma:= \{\gamma\in \mathcal{C}(\overline{D}, X)\;|\; \gamma|_{\partial D} =\operatorname{id}\},
	\end{equation*}
	then $c$ is a critical value of $I$.\\
    
	As an application, the above theorem was used in \cite{Ra} to find a weak solution of the following problem 
	\begin{equation*}
		\begin{cases}
			-\Delta u= \lambda f(x)u +p(x,u),\quad x\in \Omega\\
			u =0, \quad x\in \partial \Omega, 
		\end{cases}
		\end{equation*}
	where $\Omega \subset \mathbb{R}^N$ is a bounded domain whose boundary is a smooth manifold, $\lambda$ is an eigenvalue of the Sturm-Liouville eigenvalue problem
	
		\begin{equation*}
		\begin{cases}
			-\Delta u= \mu f(x)v,\quad x\in \Omega\\
			v =0, \quad x\in \partial \Omega, 
		\end{cases}
	\end{equation*}
	$f$ is a positive function, $p(x,u)\in \mathcal{C}(\overline{\Omega}\times\mathbb{R}, \mathbb{R})$ and satisfies some growth conditions and where the associated functional is strongly defined and satisfies the saddle point geometry. 
	
	It is crucial that $\dim Y < \infty$ before applying the above theorem.

Since the space $Y$ is finite-dimensional, the classical Saddle Point Theorem mentioned above cannot be used when one investigates, for example, the existence of a solution to the following semilinear Schrödinger equation 
\begin{equation*}
	(P)\qquad
	\begin{cases}
		-\Delta u + V(x) u = \lambda u + g_0 (x,u), \quad  x\in \mathbb{R}^N, \\
		u \in H^1(\mathbb{R}^N),
	\end{cases} \\
\end{equation*}
where $ V \in \mathcal{C}(\mathbb{R}^N, \mathbb{R})$ is $\mathbb{Z}^N-$periodic, $\lambda$ is a parameter and the non-linearity $g_0 \in \mathcal{C}(\mathbb{R}^N \times \mathbb{R}, \mathbb{R})$ satisfies some growth conditions.

Indeed, it is well known that $(P)$ has a variational structure, and its weak solutions are critical points of the functional $J$ defined on the Hilbert space $X:= H^1(\mathbb{R}^N) $ by
\begin{equation}
	\label{eq 1}
	J(u) = \dfrac{1}{2} \int_{\mathbb{R}^N} \Big(|\nabla u|^2 + V(x) u^2 -\lambda u^2\Big)dx - \int_{\mathbb{R}^N} G_0(x,u)dx
\end{equation}
with 
\begin{equation*}
	G_0(x,u) := \int_{0}^{u}g_0(x,t)dt.
\end{equation*}

We recall that $u$ is a weak solution of $(P)$ if $u \in X$, and satisfies for any $ \varphi \in X$ : 
\begin{equation*}
	\int_{\mathbb{R}^N} \Big(\nabla u \nabla \varphi + V(x) u \varphi -\lambda u \varphi \Big)dx - \int_{\mathbb{R}^N} g_0(x,u)\varphi dx =0.
\end{equation*}

If the negative space $Y$ of the quadratic part in $J$ is finite dimensional, then we have the orthogonal decomposition $X= Y \oplus Z$ with $Z$ the positive space of the quadratic part in $J$ and we can use the above classical Saddle Point Theorem of Rabinowitz to obtain the critical points of $J$, see \cite{Te,CoTe}. However, for some important class of functions $V$ and if $\lambda$ does not belong to the spectrum of $-\Delta + V$, it is well known that $Y$ is infinite dimensional and the problem $(P)$ is strongly indefinite; see \cite{LiSh}. \\

Our aim in this present paper is to establish a natural generalization of the classical Saddle Point Theorem \cite[Theorem 4.6]{Ra} (mentioned above) which may deal with more general strongly indefinite functionals. As an application, we show that problem $(P)$ admits at least one solution.

We would like to point out that there are various infinite-dimensional generalizations of the Saddle point theorem in the literature. Among others, we can mention \cite{BeRa, Ho, Silva1, Silva2, LiSh, BaCo}. Here, our approach is different. We use the $\tau-$topology introduced by Kryszweski and Szulkin to formulate a general minimax result (see Theorem 2.1). Combining the latter result with another one (see Proposition 3.1) we give a straightforward proof of our abstract result. \\

The remainder of the paper is organized as follows. In the next section, we recall some classical results and present some other preliminary results that will be used later. In Section 3, we state and prove our main result, and in Section 4, we apply this result to obtain at least one solution to problem $(P)$.

\section{A general minimax principle}
Before we present the main result of this section, which is a general minimax principle, we recall some preliminary results.
\subsection{Kryszewski-Szulkin degree theory}
In this subsection, we are following the presentation of the degree theory of Kryszewski and Szulkin given in \cite{Wi} by Michel Willem.

Let $Y$ be a real separable Hilbert space endowed with inner product $( \cdot, \cdot )$ and the associated norm $\|\cdot\|$.  

On $Y $ we consider the $\sigma-$topology introduced by Kryszewski and Szulkin \cite{KS}; that is the topology generated by the norm 
\begin{equation}
	|||u||| := \sum_{k =0}^{\infty} \frac{1}{2^ {k+1}}|(u, e_k)|, \quad u \in Y,
\end{equation}
where $(e_k)_{k\ge 0}$ is a total orthonormal sequence in $Y$.\\

\begin{remark}	
	\label{remark 1}
	By the Cauchy-Schwarz inequality, one can show that $|||u||| \le \|u\|$ for every $u \in Y$. Moreover, if $(u_n)$ is a bounded sequence in $Y$ then 
	\begin{equation*}
		u_n \rightharpoonup u \Longleftrightarrow u_n \overset{\sigma}{\rightarrow} u,
	\end{equation*}
	where $\rightharpoonup$ denotes the weak convergence and $ \overset{\sigma}{\rightarrow}$ denotes the convergence in the $\sigma-$topology.
\end{remark}
Let $U$ be an open bounded subset of $Y$ such that its closure $\overline{U}$ is $\sigma-$closed.
\begin{definition}[\cite{Wi}]
	\label{definition 2.1}
	A map $f : \overline{U} \rightarrow Y$ is $\sigma-$admissible (admissible for short) if\\
	(1) $f$ is $\sigma-$continuous, \\
	(2) each point $u \in U$ has a $\sigma-$neighborhood $N_u$ in $Y$ such that $(\operatorname{id}-f)(N_u\cap U)$ is contained in a finite-dimensional subspace of $Y$.
\end{definition}
\begin{definition}[\cite{Wi}]
	A map $h :[0,1]\times \overline{U}\rightarrow Y$ is an admissible homotopy if \\
	(1)  $0 \notin h([0,1] \times \partial U)$,\\
	(2) $h$ is $\sigma-$continuous, that is $t_n \rightarrow t$ and $u_n \overset{\sigma}{\rightarrow} u$ implies $h(t_n,u_n) \overset{\sigma}{\rightarrow} h(t,u)$,\\
	(3) $h$ is $\sigma-$locally finite-dimensional. That is, for any $(t,u) \in[0,1]\times U$ there is a neighborhood $N_{(t,u)}$ in the product topology of $[0,1]$ and $(X, \sigma)$ such that $\{ v-h(s,v) \; | \; (s,v) \in N_{(t,u)} \cap ([0,1]\times U)\}$ is contained in a finite-dimensional subspace of $Y$.
\end{definition} Then, for an admissible map $f$ such that $0 \notin f(\partial U)$   we have (see \cite[ Theorem 6.6]{Wi})
\begin{equation*}
	deg (h(0,.), U) = deg(h(1,.),U),
\end{equation*}
where $deg$ is the topological degree of $f$ (about 0). Such a degree (called Kryszewski and Szulkin degree, see \cite[Definition 6.3]{Wi}) possesses the usual properties; in particular, if $f : \overline{U} \rightarrow Y$ is an admissible map with $0 \notin f(\partial U)$ and $deg(f, U)\ne 0$, then there exists $u \in U$ such that $f(u)= 0$.\\

Now, let $X = Y\oplus Z$, where $Y$ is closed and $Z=Y^\perp$, be a real separable Hilbert space endowed with inner product $(\cdot, \cdot)$ and the associated norm $\|\cdot\|$. Let $(e_k)_{k\ge 0}$ be an orthogonal basis of $Y$. On $X$, we define a new norm by setting
\begin{equation*}
	|u|_\tau := \max \Bigg(\sum_{k =0}^{\infty} \frac{1}{2^ {k+1}}|( Pu, e_k)|, \|Qu\| \Bigg), \quad u \in X,
\end{equation*}
where $P, Q$ are, respectively, orthogonal projections of $X$ into $Y$ and $Z$ and we denote by $\tau$  the topology generated by this norm. The topology $\tau$ was introduced by Kryszewski and Szulkin \cite{KS}.
\begin{remark}	
	\label{remark 2} For every $u\in X$, 
	we have $\|Qu\| \le |u|_\tau$ and $|||Pu||| \le |u|_\tau$. Moreover, if $(u_n)$ is a bounded sequence in $X$ then 
	\begin{equation*}
	u_n \overset{\tau}{\rightarrow} u \Longleftrightarrow	Pu_n \rightharpoonup Pu \quad\text{and}\quad Qu_n \to Qu,
	\end{equation*}
	where $\rightharpoonup$ denotes the weak convergence and $ \overset{\tau}{\rightarrow}$ denotes the convergence in the $\tau-$topology.
\end{remark}

\subsection{ A deformation lemma}
We recall some definitions and standard notation : 
Let  $S \subset X$ and $J : X \rightarrow \mathbb{R}$ a functional of class $\mathcal{C}^1$. We shall denote by $J'(u)$, the Fréchet derivative of $J$ at $u\in X$. 

For every $\alpha, \beta \in \mathbb{R}$ we denote by \;$ dist(u, S):= \|u-S\|$ the distance from the point $u$ to the set $S$ (in the topology induced by the norm $\|\cdot\|$),\;
$S_{\alpha} := \{u \in X \;|\; dist(u,S) \le \alpha\} $,\; $J_{\beta}:= \{u \in X \;|\; J(u)\ge \beta \}$, \; $J^{\alpha}:= \{u \in X \;|\; J(u)\le \alpha \}$ and $J_\beta^{\alpha}:= J_\beta \cap J^\alpha$.\\

Since $X$ is a Hilbert space, then the gradient of $J$ is a continuous vector field $\nabla J : X \rightarrow X$, where for every $u \in X$, $\nabla J (u)$ is characterized by the unique element in $X$ given by the Riesz representation theorem such that 
\begin{equation*}
	J'(u) v= (\nabla J(u), v), \; \text{for every}\; v \in X.
\end{equation*}

The functional $J$ is said to be $\tau-$upper semi-continuous if the set $J_{\beta}$ is $\tau-$closed and 
we say that $\nabla J$ is weakly sequentially continuous if the sequence $(\nabla J(u_n))$ converges weakly to $\nabla J(u)$ whenever $(u_n)$ converges weakly to $u$ in $X$. \\

We consider the class of $\mathcal{C}^1-$functionals $J : X \rightarrow \mathbb{R}$ such that 

\text{(A)}	$J$ is $\tau-$upper semi-continuous and $\nabla J$ is weakly sequentially continuous.\\

We recall the following lemma which will play a key role in the proof of the main result of this section; see Lemmas \cite[Lemma 6.8]{Wi} and \cite[Lemma 8]{Baco2}.
\begin{lemma}[Deformation Lemma]
	\label{lemma 2.1}
	Assume that $J$ satisfies $(A)$. Let $S \subset X$, $c \in \mathbb{R}, \; \epsilon, \; \delta >0$ be such that 
	\begin{equation}
		\label{equation 3}
		\text{for any}\quad u \in J^{-1}\big([c-2\epsilon, c+ 2 \epsilon]) \cap S_{2\delta}, \; \text{we have}\; \| J' (u) \| \ge \frac{8\epsilon}{\delta}.
	\end{equation}
	Then there exists  $\eta \in \mathcal{C}([0,1]\times J^{c +2\epsilon}, X)$ such that  :\\
	(i) $\eta(t,u) =u $ if $t=0$ or if $u \notin J^{-1}\big([c-2\epsilon, c+ 2 \epsilon]) \cap S_{2\delta} $, \\
	(ii) $\eta(1, J^{c + \epsilon} \cap S) \subset J^{c -\epsilon}$,\\
	(iii) $\| \eta(t,u)-u\| \le \frac{\delta}{2}$ for every $u \in J^{c+2\epsilon}$ and for every $t \in [0,1]$,\\
	(iv) $J(\eta(.,u))$ is non-increasing for every $ u \in J^{c+2\epsilon}$, \\
	(v) each point $(t,u) \in [0,1]\times J^{c +2 \epsilon}$ has a $\tau-$neighborhood $N_{(t,u)}$ such that  $\{ v-\eta(s,v)\; |\; (s,v) \in N_{(t,u)}\cap \big([0,1]\times J^{c + 2\epsilon}\big)\}$ is contained in a finite-dimensional subspace of $X$, \\
	vi) $\eta$ is $\tau-$continuous.
\end{lemma}
With the aid of the Lemma $\ref{lemma 2.1}$, we prove the following general minimax principle result.
\subsection{General minimax principle}
The following theorem is an infinite-dimensional generalization of \cite[ Theorem 2.8]{Wi}.
\begin{theorem}[General minimax principle]
	\label{theorem 2.1}
Assume that $J\in \mathcal{C}^1(X, \mathbb{R})$ satisfies $(A)$, that is, $J$ is $\tau-$upper semi-continuous and $\nabla J$ is weakly sequentially continuous. Let $M$ be a closed metric subset of $X$, and let $M_0$ be a closed subset of $M$. Let $\operatorname{int}(M)$ be the interior of $M$.
Let us define 
 \begin{center}
 	$\Lambda_0 : = \Big\{\gamma_0 : M_0 \rightarrow X \;\Big|\;\gamma_0 $ \quad \text{is \;$\tau-$continuous}\Big\},
 \end{center}
\begin{center}
		$\Gamma : = \Big\{\gamma : M \rightarrow X \;\Big|\;\gamma \quad \text{is}\quad \tau-$continuous, $\gamma \big|_{ M_0} \in \Lambda_0$ and\\ every  $u\in \operatorname{int}(M)$ has a $\tau-$neighborhood $N_u$ in $X$ such that $(\operatorname{id}-\gamma)(N_u\cap \operatorname{int}(M))$ \\ is contained in a finite-dimensional subspace of $X \Big \}$.
	\end{center}
	If $J$ satisfies
	\begin{equation}
		\label{eq 4}
		\infty > c := \underset{\gamma \in \Gamma}{\inf}\; \underset{u \in M}{\sup} \;J(\gamma(u))> a := \underset{\gamma_0 \in \Lambda_0}{\sup}\; \underset{u \in M_0}{\sup} \; J(\gamma_0(u)),
	\end{equation}
	Then,
	for every $\epsilon \in ] 0, \frac{c-a}{2}[$, $\delta >0$ and $\gamma \in \Gamma$ such that
	\begin{equation}
		\label{eq 5}
		\underset{u\in M}{\sup}\; J (\gamma(u)) \le c +\epsilon , 
	\end{equation}
	there exists $u \in X$ such that \\
	a) $c-2\epsilon \le J(u) \le c+2\epsilon$,\\
	b) $dist(u, \gamma(M)) \le 2 \delta$,\\ 
	c) $\|J' (u)\| < \frac{8\epsilon}{\delta}$.
\end{theorem}
The proof of Theorem $\ref{theorem 2.1}$ follows the lines of the proof of \cite[Theorem 2.8]{Wi}. For the reader’s convenience, we sketch it here.
\begin{proof}
	Suppose that there are $\epsilon \in ] 0, \frac{c-a}{2}[$, $\delta >0$ and $\gamma \in \Gamma$ verifying
	$\underset{u\in M}{\sup}\; J (\gamma(u)) \le c +\epsilon$ 
	such that for every $u \in X$, we have 
	$c-2\epsilon \le J(u) \le c+2\epsilon$,
	$dist(u, \gamma(M)) \le 2 \delta$
	and $\|J' (u)\| \ge \frac{8\epsilon}{\delta}$.\\
	Then, the Lemma $\ref{lemma 2.1}$ can be applied with $S := \gamma(M)$. Since $\epsilon \in ] 0, \frac{c-a}{2}[$, then
	\begin{equation}
		\label{eq 6}
		c -2\epsilon > a.
	\end{equation} 
	Define $\beta : M \rightarrow X, \; \; u\mapsto \eta (1, \gamma(u))$, where $\eta$ is given by Lemma $\ref{lemma 2.1}$. Then, $\beta \in \Gamma$:\\
	$(1)$ For every $u \in M_0$, we obtain, from $(\ref{eq 4})$ and $(\ref{eq 6})$, $J(\gamma_0(u))< c-2\epsilon$. So, $\gamma_0(u) \notin J^{-1}\big( [ c-2\epsilon, c+2\epsilon] \big)$. Thus 
	\begin{equation*}
		\beta(u)= \eta(1, \gamma_0(u)) = \gamma_0(u).
	\end{equation*}
	$(2)$ By definition, $\beta$ is $\tau-$continuous. \\
	$(3)$ Let $u \in \operatorname{int}(M)$. Since $\gamma \in \Gamma$, $u$ has a $\tau-$neighborhood $N_u$ in $X$ such that  $(\operatorname{id}-\gamma)(N_u \cap \operatorname{int}(M)) \subset F_1$, where $F_1$ is a finite-dimensional subspace of $X$. From $(v)$ of the Deformation Lemma (Lemma $\ref{lemma 2.1}$), the point $(1, \gamma(u))$ has a $\tau-$neighborhood $V_{(1, \gamma(u))} := V_1 \times V_{\gamma(u)}$ with $V_1$ a neighborhood of $\{1\}$ and $V_{\gamma(u)}$ a $\tau-$neighborhood of $\gamma(u)$ such that $\{ x- \eta(s,x) \; |\; (s,x) \in V_{(1, \gamma(u))} \cap ([0,1]\times J^{c+2\epsilon})\}$ is contained in a finite-dimensional subspace $F_2$ of $X$. Thus for every $ v \in N_u \cap \gamma^{-1}(V_{\gamma(u)})\cap \operatorname{int}(M)$, we have $(\operatorname{id}-\beta)(v) = (\operatorname{id}-\gamma)(v) + \gamma (v) - \eta(1, \gamma(v)) \in F_1 + F_2$ which is also finite-dimensional of $X$.\\
	It follows from $(1)$, $(2)$ and $(3)$ that $\beta \in \Gamma$. 
	
	Since for every $u\in M$, $\gamma (u) \in J^{c+\epsilon}$, that is, $\gamma(u)\in S\cap J^{c+\epsilon}$, then by $(ii)$ of Lemma $\ref{lemma 2.1}$, $\eta(1, \gamma(u))\in J^{c-\epsilon}$. This implies that
	\begin{equation*}
		\underset{u\in M}{\sup}\; J(\beta(u)) = \underset{u\in M}{\sup}\; J\Big(\eta(1, \gamma(u)))\Big) \le c-\epsilon.
	\end{equation*}
	From $(\ref{eq 4})$ again, we finally have
	\begin{equation*}
		c \le \underset{u\in M}{\sup}\; J(\beta(u)) \le c-\epsilon,
	\end{equation*}
	giving a contradiction.\\
\end{proof}

From the preceding general minimax principle, we can deduce, under the given conditions, the existence of a Palais-Smale sequence. Typically, the existence of a critical point is then derived whenever the functional $J \in \mathcal{C}^1(X, \mathbb{R})$ also satisfies the $(PS)_c$ condition. We recall that 
$J$ is said to satisfy the $(PS)_c$ condition (or the Palais-Smale condition at level $c$) if any sequence $(u_n)\subset X$ such that 
\begin{equation*}
	J(u_n) \rightarrow c \quad \textit{and} \quad J' (u_n) \rightarrow 0
\end{equation*}
has a convergent subsequence.

\section{Generalized Saddle point Theorem}
In this Section, we will prove the following generalized Saddle Point Theorem, which is an infinite dimensional generalization of the classical Saddle Point Theorem of P. H. Rabinowitz. 

For $\rho >0$ set 
\begin{equation}
\label{eq 7}
	M = \{ u \in Y \; |\; \|u\| \le \rho \},
\end{equation}
Then $M$ is a sub-manifold of $Y$ with boundary $\partial M$.\\

Let us define $\operatorname{int}(M):= \{ u \in Y \; |\; \|u\| < \rho \}$.

\begin{theorem}[Generalized Saddle Point Theorem]
	\label{theorem 3.1}
	Assume that $J\in \mathcal{C}^1(X, \mathbb{R})$ satisfies $(A)$, that is, $J$ is $\tau-$upper semi-continuous and $\nabla J$ is weakly sequentially continuous. Suppose that
	\begin{equation}
		\label{eq 8}
		b:= \underset{Z}{\inf} \; J > a := \underset{\partial M}{\sup}\; J . 
	\end{equation}
    
Let $c\in \mathbb{R}$ be characterized as
	\begin{center}
		$c:= \underset{\gamma \in \Gamma}{\inf} \; \underset{u\in M}{\sup}\; J(\gamma(u))$, \\ 
where	$\Gamma : = \Big\{\gamma : M \rightarrow X \;\Big|\;\gamma \quad \text{is}\quad \tau-$continuous, $\gamma \big|_{ \partial M} = \operatorname{id}$ and\\ every  $u\in \operatorname{int}(M)$ has a $\tau-$neighborhood $N_u$ in $X$ such that $(\operatorname{id}-\gamma)(N_u\cap \operatorname{int}(M))$ \\ is contained in a finite-dimensional subspace of $X \Big \}$.
\end{center}	

	Then, there exists $(u_n)\subset X$ such that 
\begin{equation*}
	J(u_n) \rightarrow c, \qquad J'(u_n) \rightarrow 0, \quad \text{as}\quad n\to \infty.
\end{equation*}
\end{theorem} 
\begin{remark}
	Since, in the finite-dimensional spaces, all norms are equivalent, in case $Y$ is finite-dimensional, the conclusion remains valid (this is the classical Saddle Point Theorem of Rabinowitz \cite{Ra}). Therefore, we will only prove the theorem if $Y$ is infinite dimensional.
\end{remark}

Prior to proving Theorem~\ref{theorem 3.1}, we state some results which are needed in the proof.\\

The following results are respectively the generalizations of \cite[ Theorem 2.9]{Wi} and \cite[Corollary 4.5]{Ra}.
\begin{corollary}
	\label{cor 3.1}
	Suppose that the assumptions of Theorem $\ref{theorem 2.1}$ are satisfied and suppose that $J$ satisfies $(\ref{eq 4})$. Then there exists a sequence $(u_n)\subset X$ satisfying 
	\begin{equation*}
		J(u_n) \rightarrow c, \qquad J'(u_n) \rightarrow 0.
	\end{equation*}
	
	In particular, if $J$ satisfies $(PS)_c$ condition, then $c$ is a critical value $J$.\\
\end{corollary}

\begin{proposition}
	\label{prop 3.1}
	Let $U$ be an open bounded subset of $X$ containing $0$ such that its closure $\overline{U}$ is $\sigma-$closed and $0\notin \partial U$. If $\phi: \overline{U} \mapsto X$ is an $\sigma-$admissible map such that  $\phi =\operatorname{id}$ on $\partial U$,
	then 
	\begin{equation*}
		deg (\phi, U) = deg (\operatorname{id}, U). 
	\end{equation*}
\end{proposition}
\begin{proof}
	Consider the map
	\begin{center}
		$H : [0,1]\times \overline{U} \rightarrow X,  \quad \quad (t,u) \mapsto (1-t)u + t \;\phi(u)$.
	\end{center}
	Then, 	$H$ is an admissible homotopy. Indeed: 
	\begin{enumerate}
	\item[$(i)$] For all $(t,u)\in [0,1]\times \partial U, \; 0 \ne H(t,u)$. If not, there exists $(t,u) \in [0,1]\times \partial U$ such that $H(t,u) =0$. Then we obtain the contradiction $0 \in \partial U$.
	\item[$(ii)$]  Suppose $t_n \rightarrow t$ in $[0,1]$, $u_n\overset{\sigma}{\rightarrow }u$ in $\overline{U}$. We have
	\begin{align*}
		\begin{split}
			|||H(t_n,u_n) - H(t,u)||| &= ||| (1-t_n)u_n +t_n \phi(u_n)-(1-t)u-t\phi(u)||| \\
			&=||| (u_n -u)-u(t_n -t)-t_n(u_n -u)+t_n(\phi(u_n) -\phi(u))\\ &\qquad \qquad \qquad + \phi(u)(t_n -t)|||\\
			&\le||| u_n -u||| + |t_n -t| (|||u||| + |||\phi(u)|||) +|t_n| \; |||u_n -u||| \\ &    \qquad \qquad \qquad + |t_n| \; |||\phi(u_n) -\phi(u)||| \\
			&\underset{ n\rightarrow \infty}{ \rightarrow 0}.
		\end{split}
	\end{align*}
	That is, $H$ is $\sigma-$continuous.
	\item[$(iii)$] Let $(t,u)\in [0,1]\times U$. Then $u -H(t,u) = u -(1-t)u - t \;r(u)= t(\operatorname{id}-\phi)(u)$. Using $(2)$ of Definition $\ref{definition 2.1}$ there is a neighborhood $N_{(t,u)} := [0,1]\times N_u$ de $(t,u)$ such that $\{ v-H(s,v) : (s,v) \in N_{(t,u)} \cap ([0,1]\times U)\}$ is contained in a finite-dimensional subspace of $X$. Thus $H$ is $\sigma-$locally finite-dimensional.
	\end{enumerate}
	
Hence, by normalization and homotopy invariance, we have
	\begin{equation*}
		deg (\phi, U)=deg (H(1,.), U) = deg (H(0,.), U) =deg (\operatorname{id}, U) =1.
	\end{equation*}
\end{proof}
\subsection{Proof of Theorem $\ref{theorem 3.1}$}
Observe that, $M$ as in $(\ref{eq 7})$ is $\sigma-$closed. In order to apply Corollary $\ref{cor 3.1}$, we have only to verify that $c\ge b$. Denote by $P$ the projection of $X$ onto $Y$.
Let $\gamma \in \Gamma$ and consider the map 
\begin{eqnarray*}
	P\gamma &:& M \rightarrow Y\\  u &\mapsto& P(\gamma(u)).
\end{eqnarray*}
\begin{enumerate}
	\item[(i)] The map $P\gamma$ is $\tau-$continuous. In fact, let $u_n \overset{\tau}{\rightarrow} u$, then $\gamma(u_n)\overset{\tau}{\to} \gamma(u)$ and $|P\gamma(u_n)-P\gamma(u)|_\tau =|P(\gamma(u_n)-\gamma(u))|_\tau \le |\gamma(u_n)-\gamma(u)|_\tau\to 0$, as $n\to \infty$.
	\item[(ii)] Let $u\in \partial M$. Then $P(\gamma(u)) =Pu =u \ne 0$.
	\item[(iii)] Let $u\in \operatorname{int}(M)$. Then $u$ has a $\tau-$neighborhood $N_u$ in $X$ such that $(\operatorname{id}-\gamma)(N_u\cap U) \subset E_0$, where $E_0$ is a finite-dimensional subspace of $X$. Let $v\in N_u\cap \operatorname{int}(M)$. We have $(\operatorname{id}-P\gamma)(v)= P(v-\gamma(v))\in E_0$.
\end{enumerate}

Since the topology $\sigma$ is induced by the topology $\tau$ on $Y$, we conclude  that $ P\gamma : M \rightarrow Y$ is  $\sigma-$admissible and $0\notin P\gamma (\partial M)$. 

Then the Kryszewski-Szulkin degree of $P\gamma$ is well defined. Hence, by normalization and homotopy invariance and using the Proposition $\ref{prop 3.1}$ we have
\begin{equation*}
	deg (P\gamma, \operatorname{int}(M)) = deg (\operatorname{id}, \operatorname{int}(M)) =1. 
\end{equation*}

By existence, there is $ x \in \operatorname{int}(M)$ such that $P\gamma(u) =0$. Consequently for each $\gamma \in \Gamma$, there exists $x=x(\gamma) \in \operatorname{int}(M)$ such that $\gamma(x)= (\operatorname{id}-P)\gamma (x) \in Z$. 

It follows that for every $\gamma \in \Gamma$ 
\begin{equation*}
	b:=	\underset{u \in Z}{\inf} J (u) \le J(\gamma(x)) \le \underset{u \in M}{\sup}\; J(\gamma(u)).
\end{equation*}

Hence, $c\ge b$ and applying Corollary $\ref{cor 3.1}$, the conclusion of Theorem $\ref{theorem 3.1}$ follows. The proof of Theorem $\ref{theorem 3.1}$ is thus complete. \quad $\square$

\section{Semilinear Schrödinger equation}

In this section we apply our abstract theorem to obtain at least one solution to the following indefinite semilinear Schrödinger equation 

\begin{equation*}
	(P) \qquad 
	\begin{cases}
		-\Delta u + V(x) u = \lambda u + g_0 (x,u), \quad  x\in \mathbb{R}^N, \\
		u \in H^1(\mathbb{R}^N),
	\end{cases} 
\end{equation*}
where $\lambda$ is a parameter. Our assumptions on $(P)$ are the following : 
\begin{enumerate}
	\item[$(V)$] $V \in \mathcal{C}(\mathbb{R}^N, \mathbb{R})$ is $\mathbb{Z}^N-$periodic, that is $V(x+z)=V(x)$ for all $x\in \mathbb{R}^N$ and $z\in \mathbb{Z}^N$, $\lambda\notin \spec (-\Delta +V)$, the spectrum of $-\Delta + V$;
	\item[$(g_0)$]  $g_0 \in \mathcal{C}(\mathbb{R}^N \times \mathbb{R}, \mathbb{R})$. For every $\epsilon >0$ there exists $0\le b_\epsilon(x)\in L^2(\mathbb{R}^N)$ such that $0 \le g_0(x,s)\le b_\epsilon(x) + \epsilon |s|$ a.e. $x\in \mathbb{R}^N$, $s \in \mathbb{R}$.
\end{enumerate}
\begin{remark}
	We would like to mention that 
	\begin{enumerate}
		\item[$(1)$] the assumption $(V)$ is identical to the one considered in \cite{LiSh} and the growth condition $(g_0)$ is inspired from \cite{Te}; and 
		\item[$(2)$] the zero function $u=0$ is a solution of problem $(P)$ only when $g_0(x,0)=0$, which is not the case for infinetely many nonlinearities like, for instance, the following one.
	\end{enumerate}
\end{remark}
\begin{example}
	The set of functions $g_0$ is nonempty. Indeed, define \[g_0(x,s)= e^{-|x|^2-|s|^2}, \quad \text{for every}\quad x\in \mathbb{R}^N,\quad \text{and every}\quad s\in \mathbb{R}.\] Then, the function $g_0$ satisfies assumption $(g_0)$.
\end{example}

Throughout this section, $|\cdot|_p$ stands for the usual $L^p$ norm.
We denote $\rightarrow$ (resp. $\rightharpoonup$) the strong convergence (resp. the weak convergence).\\

For the proof of the following result, see \cite{Wi} or \cite{LiSh}.
\begin{lemma}
Under assumptions $(V)$ and $(g_0)$, the functional $(\text{given in $(\ref{eq 1})$})$ $J \in \mathcal{C}^1(X:= H^1(\mathbb{R}^N), \mathbb{R})$  with
	\begin{equation}
		\label{eq 9}
		J'(u) \varphi= \int_{\mathbb{R}^N} \Big(\nabla u \nabla \varphi + V(x) u \varphi -\lambda u \varphi \Big)dx - \int_{\mathbb{R}^N} g_0(x,u)\varphi dx  \quad \text{for every} \quad \varphi \in H^1(\mathbb{R}^N).
	\end{equation}
\end{lemma}

Here is the main result of this section.

\begin{theorem}
	\label{theorem 4.1}
	Suppose that assumptions $(V)$ and $ (g_0)$ hold, then $(P)$ has at least one solution. 
\end{theorem}

\begin{proof}
	Let $L$ be the self-adjoint operator $L: H^1(\mathbb{R}^N) \rightarrow H^1(\mathbb{R}^N)$ defined by 
	\begin{equation*}
		( Lu,v)_1 := \int_{\mathbb{R}^N} (\nabla u . \nabla v + V(x)uv -\lambda uv)dx,
	\end{equation*}
	where $(\cdot)_1 $ is the usual inner product in $H^1(\mathbb{R}^N)$. By assumption $(V)$, it is well known that  $X= Y\oplus Z$ with
	$Y$ and $Z$ two infinite dimensional orthogonal subspaces on which $L$ is respectively negative definite and positive definite (see \cite{Wi} or \cite{LiSh}). We denote $P: X \rightarrow Y$ and $Q: X \rightarrow Z$ the orthogonal projections. 
	
	We introduce a new inner product on X (equivalent to $(\cdot)_1$, see \cite{KS}) by the relation 
	\begin{equation*}
		\langle u,v \rangle:= ( L(Qu-Pu), v )_1 ,\quad u, v \in X,
	\end{equation*}
	with the corresponding norm
	\begin{equation*}
		\|u\| := \langle u,u \rangle^{\frac{1}{2}},
	\end{equation*}
	which is equivalent to the standard norm $\|\cdot\|_1$ on $H^1(\mathbb{R}^N)$.
	
	Since the inner products $\langle \cdot \rangle$ and $(\cdot)_1$ are equivalent, $Y$ and $Z$ are also orthogonal with respect to $\langle \cdot \rangle$. 
	
	One can verify that the
	functional $J$ given in  $(\ref{eq 1})$ can be  written as 
	\begin{equation}
		\label{eq 10}
		J(u) = \dfrac{1}{2}\Big(\|Qu\|^2 - \|Pu\|^2 \Big)- \int_{\mathbb{R}^N} G_0(x,u)dx\quad \text{for every}\quad u \in H^1(\mathbb{R}^N),
	\end{equation}
	and 
	\begin{equation*}
	\langle \nabla J(u), v \rangle = \langle Qu, v\rangle - \langle Pu, v\rangle - \int_{\mathbb{R}^N} g_0(x,u) vdx.
	\end{equation*}
	
	First, let us verify that $J$ is $\tau-$upper semi-continuous, $\nabla J$ is weakly sequentially continuous.\\

	1. $\nabla J$ \textbf{is weakly sequentially continuous}. Let $(u_n) \subset X$ such that $u_n \rightharpoonup u $. Then $ u_n \rightarrow u $ in $ L_{loc}^2 (\mathbb{R}^N)$. Let us consider an arbitrary subsequence still denote by the same symbol $(u_n)$. We can assume that $u_n\rightarrow u$ a.e. on $\mathbb{R}^N$. Since $b_\epsilon(\cdot) \in L_{loc}^2 (\mathbb{R}^N)$, by the assumption $(g_0)$, we have $g_0(x,u_n)\rightarrow g_0(x,u)$ in $L_{loc}^2 (\mathbb{R}^N)$ and $ g_0(x,u_n)\varphi \rightarrow g_0(x,u)\varphi $ in $L_{loc}^2 (\mathbb{R}^N) $ for every $\varphi \in X$. 
	Thus, $\langle \nabla J(u_n), \varphi\rangle \rightarrow \langle \nabla J(u), \varphi\rangle$ for every $\varphi\in X$.\\
	
	2.  \textbf{ $J$ is $\tau-$upper semi-continuous}.
	For every $c\in \mathbb{R}$, let us show that the set $\{u\in X\,|\, J(u)\ge c\}$ is $\tau-$closed. 
    
	Let $(u_n) \subset X$ such that $u_n\overset{\tau}{\rightarrow} u$ in $X$ and $c\le J(u_n)$. Then, $(u_n)$ is bounded. Thus, there is a subsequence $(u_{n_k})$ such that
	\begin{equation*}
		Pu_{n_k} \rightharpoonup Pu, \qquad Qu_{n_k} \rightharpoonup Qu, \qquad u_{n_k} \rightarrow u \; \text{in} \;L_{loc}^2 (\mathbb{R}^N).
	\end{equation*}
	
	Up to a subsequence we also have $u_{n_k} \rightarrow u$ a.e. on $\mathbb{R}^N$. Since $G_0(x,u_{n_k})\ge 0$, then by Fatou lemma we have 
	\begin{equation*}
		\int_{\mathbb{R}^N} G_0(x,u) dx= \int_{\mathbb{R}^N} \underset{k\rightarrow \infty}{\underline{\lim}}G_0(x,u_{n_k})dx  \le \underset{k\rightarrow \infty}{\underline{\lim}} \int_{\mathbb{R}^N} G_0(x,u_{n_k})dx.
	\end{equation*}
	
    Since $\|\cdot\|$ is weak lower semi-continuous, we have 
	\begin{equation*}
		\|Pu\|^2 \le \underset{k\rightarrow \infty}{\underline{\lim}} \|Pu_{n_k}\|^2.
        \end{equation*}
       	Moreover, since $\|Qu_{n_k}\|^2 \to \|Qu\|^2$, we have 
		\begin{equation*}
		 \underset{k\rightarrow \infty}{\overline{\lim}} \Big( -\|Qu_{n_k}\|^2\Big)= 
			-\|Qu\|^2 = 	\underset{k\rightarrow \infty}{\underline{\lim}} \Big(-\|Qu_{n_k}\|^2\Big).
		\end{equation*}
		
	Hence, 
	\begin{align*}
		-J(u) = \dfrac{1}{2}\Big( \|Pu\|^2-\|Qu\|^2 \Big) +\Psi (u) &\le \underset{k\rightarrow \infty}{\underline{\lim}} \Bigg(\dfrac{1}{2}\Big( \|Pu_{n_k}\|^2-\|Qu_{n_k}\|^2 \Big) +\Psi (u_{n_k})  \Bigg)\\
		&=  \underset{k\rightarrow \infty}{\underline{\lim}} (-J(u_{n_k})) \\
		&= - \underset{k\rightarrow \infty}{\overline{\lim}} J(u_{n_k})\\
		&\le -c.
	\end{align*}

	Now, let us verify the geometric assumption $(\ref{eq 8})$ of Theorem $\ref{theorem 3.1}$.
    
	Since the norm $\|\cdot\|$ is equivalent to $\|\cdot\|_1$ on $H^1(\mathbb{R}^N)$, and $H^1(\mathbb{R}^N)$ embeds continuously into $L^2(\mathbb{R}^N)$, there exists $S>0$ (called, by a slight abuse of terminology, the optimal constant of the Sobolev inequality; see \cite{LiSh}) such that
	\begin{equation*}
		S |u|_{2}^2 \le \|u\|^2, \quad \text{for all} \; u \in H^1(\mathbb{R}^N).
	\end{equation*}
	
	For  $u \in Z$, and choosing $\epsilon$ such that $\frac{1}{2}-\frac{\epsilon}{S}\ge \frac{1}{6}$, we have
	\begin{align*}
		J(u) &= \dfrac{1}{2} \|u\|^2 -\int_{\mathbb{R}^N} G_0(x,u)dx\\
		&\ge \dfrac{1}{2} \|u\|^2 - \int_{\mathbb{R}^N} \Big(|b_\epsilon(x)||u| + \epsilon |u|^2 \Big)dx \\
		&\ge \dfrac{1}{2} \|u\|^2 -|b_\epsilon (\cdot)|_2|u|_2 - \epsilon |u|_{2}^2 \\
		&\ge \dfrac{1}{2} \|u\|^2 - \dfrac{\epsilon}{S} \|u\|^2 -S^{-1/2}|b_\epsilon(\cdot)|_2\|u\|\\
		&\ge \dfrac{1}{6} \|u\|^2  - S^{-1/2}|b_{\epsilon}(\cdot)|_2 \|u\|.
	\end{align*}
	Thus $\underset{Z}{\inf}\; J > -\infty$. 
	
	For $u \in Y$, we have
	\begin{align*}
		J(u) &= -\dfrac{1}{2} \|u\|^2 -\int_{\mathbb{R}^N} G_0(x,u)dx\\
		&\le - \dfrac{1}{2} \|u\|^2 +|b_\epsilon(\cdot)|_2|u|_2 \\
		&\le -\dfrac{1}{2} \|u\|^2  + S^{-1/2}|b_{\epsilon}(\cdot)|_2 \|u\|.
	\end{align*}
	
	So for $\rho$ large enough, we have
	\begin{equation*}
		\underset{Z}{\inf}\; J > \underset{\partial M}{\sup}\; J,
	\end{equation*}
	where $M$ is defined in $(\ref{eq 7})$.
	
	By Theorem $\ref{theorem 3.1}$, for some $c\in \mathbb{R}$, there exists $(u_n)\subset X$ such that \begin{equation*}
		J(u_n) \rightarrow c, \qquad J'(u_n) \rightarrow 0, \quad \text{as}\quad n\to \infty.
	\end{equation*}
	
	Let us show that the sequence $(u_n)$ is bounded.
	
	For $u\in X$, we denote by $u^Y$, $u^Z$, the orthogonal projections of $u$ onto $Y$ and $Z$, respectively. 
	
	Let $\overline{u}_n = u_n^Z-u_n^Y$. Then
	\begin{align*}
		\|\overline{u}_n\| &=\int_{\mathbb{R}^N} \Big(|\nabla \overline{u}_n|^2 + |\overline{u}_n|^2\Big)dx\\
		&= \|u_n^Z\|^2+\|u_n^Y\|^2\\
		&=\|u_n\|^2.
	\end{align*} 
	
		For sufficiently large $n\in \mathbb{N}$, we have $\|J'(u_n)\| \le 1$. Hence, for such $n$
	\begin{equation*}
		\|u_n\| =\|\overline{u}_n\| \ge J'(u_n)\overline{u}_n,
	\end{equation*}
	where by $(\ref{eq 9})$,
	\begin{equation*}
	J'(u_n) \overline{u}_n = \int_{\mathbb{R}^N} \Big(\nabla u_n \nabla \overline{u}_n + V(x) u_n \overline{u}_n -\lambda u_n \overline{u}_n \Big)dx - \int_{\mathbb{R}^N} g_0(x,u_n)\overline{u}_n dx. 
	\end{equation*}
	
	Using $(\ref{eq 10})$, we have 
	\begin{align*}
		\|u_n\|&\ge \|u_n^Z\|^2+\|u_n^Y\|^2- \int_{\mathbb{R}^N} g_0(x,u_n)(u_n^Z-u_n^Y)dx\\
		&=\|u_n\|^2 - \int_{\mathbb{R}^N} g_0(x,u_n)(u_n^Z-u_n^Y)dx.
	\end{align*}
	
	By assumption $(g_0)$, and choosing $\epsilon$ such that $\epsilon\le\frac{S}{3}$, we have 
	\begin{align*}
	\Big|\int_{\mathbb{R}^N} g_0(x,u_n)(u_n^Z-u_n^Y)dx\Big|&\le  \int_{\mathbb{R}^N}|g_0(x,u_n)|(|u_n^Z|+|u_n^Y|)dx\\
	&\le \int_{\mathbb{R}^N} \Big(|b_\epsilon(x)|+\epsilon |u_n|\Big) \Big(|u_n^Z|+|u_n^Y|\Big)dx\\ 
	&\le \Big(|b_\epsilon(\cdot)|_2 +\epsilon |u_n|_2\Big)\Big(|u_n^Z|_2+|u_n^Y|_2\Big)dx\\
	&\le 2|u_n|_2\Big(|b_\epsilon(\cdot)|_2 +\epsilon |u_n|_2\Big)\\
	 &\le2|u_n|_2|b_{\epsilon}(\cdot)|_2 +\frac{2S}{3}|u_n|_2^2\\
	&\le \dfrac{2}{3}\|u_n\|^2+2S^{-1/2}|b_{\epsilon}(\cdot)|_2\|u_n\|.
	\end{align*}
	
	Hence,
	\begin{equation*}
		\|u_n\|\ge \frac{1}{3}\|u_n\|^2 -2S^{-1/2}|b_{\epsilon}(\cdot)|_2\|u_n\|.
	\end{equation*}
	
	It follows that $(u_n)$ is bounded in $X$. Up to a subsequence, we may assume that $u_n \rightharpoonup u$ in $X$. For any $w\in C_{0}^{\infty}(\mathbb{R}^N)$, since $ u_n \rightarrow u $ in $ L_{loc}^2 (\mathbb{R}^N)$, we deduce
	\begin{equation*}
		J'(u)w=\lim\limits_{n\to \infty} J'(u_n)w=0.
	\end{equation*}
	
	We conclude that $J'(u)=0$, that is, $u$ is a solution of problem $(P)$ and is nontrivial provided $g_0(x,0) \neq 0$.
	\end{proof}
\section*{Conclusion}
In this work, we presented a natural generalization of the Saddle Point Theorem. The $\sigma-$topology (the topology $\sigma$ is induced by the topology $\tau$ on $Y$) allowed us to extend \cite[Corollary 4.5]{Ra} which played a key role in the proof of the above mentioned generalization. The same ideas will be used in future work to generalize other critical point theorems for strongly indefinite functionals.

\end{document}